\documentclass[11pt,reqno]{amsart}

\DeclareRobustCommand{\rchi}{{\mathpalette\irchi\relax}}
\newcommand{\irchi}[2]{\raisebox{\depth}{$#1\chi$}}

\theoremstyle{plain}
\newtheorem{theorem} {Theorem}[section]
\newtheorem{lemma}[theorem] {Lemma}
\newtheorem{proposition}[theorem] {Proposition}
\newtheorem{corollary}[theorem] {Corollary}
\newtheorem{property}[theorem] {Property}

\newtheorem{definition}[theorem] {Definition}

\newtheorem{example} [theorem]{Example}

\theoremstyle{remark}
\newtheorem{remark}[theorem] {Remark}

\numberwithin{equation}{section}

\usepackage{amsfonts}
\usepackage{amssymb}
\usepackage{amscd}
\usepackage{xfrac,bigints}

\newcommand{\R}{{\mathbb R}}

\newcommand{\N}{{\mathbb N}}
\newcommand{\PPP}{{\mathbb P}}
\newcommand{\PP}{{\mathcal P}}

\newcommand{\LL}{{\mathcal L}}

\newcommand{\CC}{{\mathbb C}}

\newcommand{\al}{{\alpha}}
\newcommand{\la}{{\lambda}}
\newcommand{\sa}{{\sigma}}

\newcommand{\iy}{{\infty}}
\newcommand{\vphi}{{\varphi}}
\newcommand{\vep}{{\varepsilon}}
\newcommand{\g}{{\gamma}}
\newcommand{\de}{{\delta}}
\newcommand{\Om}{{\Omega}}

\newcommand{\z}{{\zeta}}
\newcommand{\be}{{\beta}}
\newcommand{\G}{{\Gamma}}

\newcommand{\bna}{\begin{eqnarray}}
\newcommand{\ena}{\end{eqnarray}}
\newcommand{\ba}{\begin{eqnarray*}}
\newcommand{\ea}{\end{eqnarray*}}
\newcommand{\beq}{\begin{equation}}
\newcommand{\eeq}{\end{equation}}

\begin{document}

\title[On Asymptotic Properties of Certain $B$-splines]
{On Asymptotic Properties of Certain $B$-splines
 in Terms of Theta-like Functions}
\author{Michael I. Ganzburg}
 \address{212 Woodburn Drive, Hampton,
 VA 23664\\USA}
 \email{michael.ganzburg@gmail.com}
 \dedicatory{Dedicated to Professor Yuri Abramovich Brudnyi on
 the occasion of his 90th birthday}

 \begin{abstract}
 The asymptotic behavior
of the Mellin transform of the associated $B$-splines $B_N^*(t)
:=t^{-N}B_N(t)$
 with special knots
in terms of theta-like functions is found.
The proof is based on
  polynomial interpolation of power functions
   and properties of certain theta-like functions.
   Pointwise asymptotics of $B_N^*$ and $B_N$
   are discussed as well.
 \end{abstract}
 \maketitle

 \section{Introduction}\label{S1}
\setcounter{equation}{0}

In this paper, we discuss the asymptotic behavior
of $B$-splines with special knots
in terms of theta-like functions.

\subsection{Asymptotics for $B$-splines}\label{SS1.1}
 $B$-splines $B_N(t)=B_N\left(t,\Om\right),\,t\in\R$,
with sets $\Om$ of $N+2$ knots
  were introduced by Curry and Schoenberg in seminal paper
\cite{CS1966} (see Eq. \eqref{E3.1} for the definition of $B_N$). $B$-splines play
an important role in numerical analysis, wavelets, and approximation theory
(see, e.g., \cite{S1971,C1992,B2001,S2007}) and,
in addition, they serve as Peano-like kernels
in integral representations of
interpolation differences \cite[Eq. (1.5)]{CS1966}.

The asymptotic behavior of $B$-spline distributions
 was described
in \cite[Theorem 6]{CS1966}
(see also \cite{CS1947} and \cite[Theorem 10.5.1]{K1968})
in the following form:
\begin{theorem}\label{T1.1}
A monotone  function $F:\R\to [0,\iy)$, with $F(-\iy)=0$ and $F(\iy)=1$,
is a limit of $B$-spline distributions
$F_N(x):=\int_{-\iy}^x B_N\left(t,\Om\right)dt$
as $N\to\iy$ for all x at which $F$ is
continuous if and only if
$F(x)=\int_{-\iy}^x \Lambda(t)dt$, where $\Lambda$
is a P\'{o}lya frequency function
(see \cite[Sect. 7]{K1968} or \cite[Sect. 10.5]{CS1966} for equivalent definitions).
\end{theorem}
In particular, since the characteristic function $\rchi_{[\de,\iy)}$ is
 a P\'{o}lya distribution function (see \cite[Sect. 10.5]{CS1966}),
 there exists a sequence of $B$-spline distributions $F_{N,\de}$ such that
 $\lim_{N\to\iy}F_{N,\de}(x)=\rchi_{[\de,\iy)}(x),\,\de\in\R$
 (see also Example \ref{Ex1.3a} with $\de=0$).

 Due to the criterion of Theorem \ref{T1.1}, if
 $\lim_{N\to\iy}F_N(x)=F(x)$
 for all x at which $F$ is continuous,
 then  $F(x)=\int_{-\iy}^x \Lambda(t)dt$, where $\Lambda$
is a P\'{o}lya frequency function.
The following result on the asymptotic behavior of $B$-splines
was established in \cite[Theorem 7]{CS1966}
(see also \cite[Theorem 10.5.2]{K1968}):
\begin{theorem}\label{T1.2} The following statements hold:\\
(i) If $F\ne\rchi_{[\de,\iy)},\,\de\in\R$, and $\Lambda\ne \Lambda_0$,
where
\ba
\Lambda_0(t):=\left\{\begin{array}{ll}
\left\vert\de_1\right\vert^{-1} e^{-t/\de_1+\de/\de_1-1},&t/\de_1\ge\de/\de_1-1,\\
0,&t/\de_1<\de/\de_1-1,
\end{array}\right.
\qquad \de\in\R,\quad\de_1\ne 0,
\ea
then
\beq\label{E1.1cc}
\lim_{N\to\iy}B_N(t,\Om)=\Lambda(t)
\eeq
uniformly on $\R$;\\
(ii) if $\Lambda=\Lambda_0$, then \eqref{E1.1cc} holds uniformly in t outside of an arbitrarily small neighborhood
of the point $t = \de- \de_1$, where $\Lambda_0$ is discontinuous;\\
(iii) if $F=\rchi_{[\de,\iy)},\,\de\in\R$, then $\lim_{N\to\iy}B_N(t,\Om)=0$
uniformly in t outside of an arbitrarily small neighborhood
of the point $t = \de$.
\end{theorem}

Theorem \ref{T1.2} was used in \cite[Sect. 9]{CS1966} and \cite[Sect. 4]{S1971}
to find the asymptotic behavior of certain $B$-splines
(see Examples \ref{Ex1.1a}, \ref{Ex1.2a}, and  \ref{Ex1.3a}).
\begin{example}\label{Ex1.1a}
{Equidistant knots},
$\Om=\left\{k-(N+1)/2: 0\le k\le N+1\right\}$.
The following
limit relation in terms of the Gaussian function holds
 for the scaled cardinal $B$-spline:
  \beq\label{E1.1c}
 \lim_{N\to\iy}\sqrt{(N+1)/12}\,\,
 B_{N}\left(\sqrt{(N+1)/12}\,t,\Om\right)
 =(2\pi)^{-1/2}\exp\left(-t^2/2\right)
 \eeq
 uniformly on $\R$.
 Relation \eqref{E1.1c} has been known at least since the early 1900s
 (see, Sommerfeld \cite{S1904}, Tricomi \cite{T1933}, Curry-Schoenberg \cite[Example 9.4]{CS1966}, and
 Unser et al. \cite{UAE1992}).
 This relation was extended by Brinks
 \cite[Theorem 2]{B2008}
 to the derivatives of $B_{N}$.
 Xu and Wang \cite[Theorem 3.1]{XW2011}
 found the convergence orders of the approximation processes
  for the derivatives of $B_{N}$ in \eqref{E1.1c}.
\end{example}
\begin{example}\label{Ex1.1aa} Not too concentrated knots.
Let $\Om=\left\{t_k:  0\le k\le N+1\right\} \subset\R$, where
the knots satisfy the following two conditions: 1.
$\sum_{k=0}^{N+1}t_k=0$ and 2.
$\lim_{N\to\iy}\left(\sum_{k=0}^{N+1}\left\vert t_k\right\vert^3\right)/
\left(\sum_{k=0}^{N+1} t_k^2
\right)^{3/2}=0$.
Using probabilistic techniques, Rzeszut and Wojciechowski \cite[Theorem 1]{RW2024} extended \eqref{E1.1c}
to the wide class of  sets $\Om$ in the following form:
\beq\label{E1.1c0}
 \lim_{N\to\iy}\,\,
 \frac{\left(\sum_{k=0}^{N+1} t_k^2
\right)^{1/2}}{N+2}\,\,
 B_{N}\left(\frac{\left(\sum_{k=0}^{N+1} t_k^2
\right)^{1/2}}{N+2}\,\,t,\Om\right)
 =(2\pi)^{-1/2}\exp\left(-t^2/2\right)
 \eeq
uniformly on $\R$. In addition, the authors generalized and strengthened results from
\cite{B2008} and \cite{XW2011} by extending
\eqref{E1.1c0} to the derivatives of $B_{N}$ and finding
the convergence orders of the approximation processes with power weights
  for the derivatives of $B_{N}$ in \eqref{E1.1c0}.
\end{example}
\begin{example}\label{Ex1.2a}
{Knots of reciprocals},
$N=2\nu,\,\nu\in\N,\,\Om=\left\{1/(2k-1): -\nu\le k\le \nu+1\right\}$.
The following relation in terms of a non-Gaussian function holds:
\beq\label{E1.1c1}
\lim_{\nu\to\iy}\frac{1}{\nu+1}B_{2\nu}\left(\frac{t}{2\nu+2},\Om\right)=\frac{2}{\pi \cosh t}
\eeq
uniformly on $\R$
(see \cite[Example 9.5]{CS1966}).
Comparing \eqref{E1.1c} and \eqref{E1.1c1}, the authors of \cite{CS1966} noted that
the knots density in \eqref{E1.1c1} is higher near the origin, and
as the result, $1/\cosh t$
decays so much more slowly than $\exp\left(-t^2/2\right)$.
In addition, it is easy to see that condition 2. of Example \ref{Ex1.1aa} is not satisfied for these knots.
\end{example}
\begin{example}\label{Ex1.3a}
{Chebyshev knots},
$\Om=\left\{\cos \frac{k \pi}{N+1}: 0\le k\le N+1\right\}$.
The knots coincide with the
zeros of the polynomial $(x^2-1)U_N(x)$ of degree $N+2$,
where $U_N$ is the Chebyshev polynomial of the second kind.
In particular, $B_N$ is a perfect $B$-spline on $[-1,1]$
(see Schoenberg \cite[Theorem 1]{S1971}).
Schoenberg \cite[Theorem 4]{S1971} also showed that
\beq\label{E1.1c2}
\lim_{N\to\iy} B_N\left(t,\Om\right)=0,\qquad t\in \R\setminus D,
\eeq
uniformly in t outside of an arbitrarily small neighborhood $D$
of the point $t = 0$.
In addition, the knots satisfy both conditions of Example \ref{Ex1.1aa} since
$\sum_{k=0}^{N+1}\cos^2 \frac{k \pi}{N+1}=(N+3)/2$ (see, e.g., \cite[Eq. 1.351(2)]{GR1994}).
Therefore, the following Gaussian asymptotic  is a consequence
of \eqref{E1.1c0}:
\beq\label{E1.1c3}
\lim_{N\to\iy}\,\,
 \frac{2^{-1/2}(N+3)^{1/2}}{N+2}\,\,
 B_{N}\left(\frac{2^{-1/2}(N+3)^{1/2}}{N+2}\,\,t,\Om\right)
 =(2\pi)^{-1/2}\exp\left(-t^2/2\right).
\eeq
Hence, \eqref{E1.1c3} describes the asymptotic behavior of $B_N\left(t,\Om\right)$
in a neighborhood of the origin. Both asymptotics \eqref{E1.1c2} and \eqref{E1.1c3} verify convergence of
the $B_N\left(t,\Om\right)$ to the delta function as $N\to \iy$ (cf. Theorem \ref{T1.2} (iii)).
\end{example}

However, there are only a few examples (including relations \eqref{E1.1c1} and \eqref{E1.1c2})
of sequences of $B$-splines that converge to a non-Gaussian function.

  In this paper, we prove
   limit relations of the form:
  \beq\label{E1.1a}
  \lim_{N\to\iy}A_N M\left(B_N^*\left(C_N\cdot,\Om\right),s\right)
  =M(\varTheta(\cdot),s)
  \eeq
  for so-called "associated" $B$-splines $B_N^*(t,\Om):=t^{-N} B_N(t,\Om)$
  with special sets of knots
  that include the squared equidistant knots and
  the squared zeros of the Gegenbauer and Hermite polynomials
  (see Theorem \ref{T5.1}).
  Here, $M(f,s)$ is the Mellin transform of
  $f,\,\varTheta$ is a theta-like function defined in Sect. 4,
   and $A_N,C_N$ are constants
  given explicitly.

  In certain cases,  pointwise asymptotic relations of the form
  \beq\label{E1.1b}
  \lim_{N\to\iy}A_N B_N^*\left(C_N t,\Om\right)
  =\varTheta(t),\qquad t\in [0,\iy),
  \eeq
  and the corresponding asymptotics for $B_N$ follow from \eqref{E1.1a}
  (see Theorem \ref{T5.2}). In particular, we prove \eqref{E1.1b} for
  $B$-splines with Chebyshev-like knots (see Corollaries \ref{C5.5}
  and \ref{C5.6})
  and find the asymptotic behavior of the perfect $B$-spline from Example \ref{Ex1.3a}
  in a right neighborhood of the point $-1$.
  (see Corollary \ref{C5.6} and Remark \ref{R5.7}).

  The main results of the paper, along with their proofs, are given
  in Section \ref{S5}. The proof of \eqref{E1.1a} of Theorem \ref{T5.1} is based on
  polynomial interpolation of power functions developed by the author
  in \cite{G2013} (see Section \ref{S2}), properties of $B$-splines and
  associated $B$-splines (see Section \ref{S3}),
   and properties of certain theta-like functions (see Section \ref{S4}).

   To derive \eqref{E1.1b} of Theorem \ref{T5.2} from \eqref{E1.1a}
   and also to prove Corollary \ref{C5.5},
   we need certain
   inequalities for the Fourier and Mellin transforms of analytic functions
   and some estimates of the Chebyshev polynomials (see Section \ref{S6}).
   In addition, Section \ref{S6} contains certain standard facts about the
   Mellin transform.
Notation and some preliminaries are given below.

\subsection{Notation and Preliminaries}\label{SS1.2}
Let $\N:=\{1,\,2,\ldots\},\,\R$ be the
set of all real numbers,
$\CC=\R+i\R$ be the set of all complex numbers,
and $\PP_n$ be the set
of all univariate algebraic polynomials with
real coefficients of degree at most $n$.
Next, we denote by $C(I)$ the set of all
complex-valued continuous
functions on a finite or infinite interval
$I$ and by
$C^{(M)}(I),\,M\in\N$, the set of all complex-valued
$M$-times continuously
differentiable functions on $I$.
Next, let $L_1(I)$ be the space of all integrable
complex-valued functions $F$
 on an interval $I$ with the finite norm
 $\|F\|_{L_1(I)} :=\int_I\vert F(y)\vert dy$.

In addition, we  need the set $\LL$ of all rapidly decreasing
functions $F\in C([0,\iy))$, satisfying
the condition $\lim_{t\to\iy}t^\eta F(t)=0$
for all $\eta\in[0,\iy)$.
Note that if  $F\in \LL$, then the function
$t^\eta F(t)$ is uniformly continuous
on $[0,\iy)$ and belongs to $L_1([0,\iy))$
for all $\eta\in[0,\iy)$.

We also use the following special functions:
the gamma function $\G(s)$,
the upper incomplete gamma function $\G(s,\tau),
\,\tau\in[0,\iy)$,
the Riemann zeta function $\z(s)$,
the Dirichlet beta function $\be(s)$,
and the theta functions $\vartheta_1(z,s)$ and
$\vartheta_4(z,s),\,z\in\CC,\,s\in\CC$.

Throughout the paper, $N\in\N,\,s\in\CC$, and the index $d$ takes two values: either $d=0$ or $d=1$.
We also use the following notation:
 \beq\label{E1.1}
 h_d(z):=\left\{\begin{array}{ll}
 \cosh z, &d=0,\\
 \sinh z, &d=1,\end{array}\right.
 \qquad z\in\CC.
 \eeq

 Next, let $f\in C([A,B])$ and let $R_{n+1}\in\PP_{n+1}$ be a polynomial with
 zeros $z_k\in [A,B],\,1\le k\le n+1$, of multiplicity $1$.
 The Lagrange interpolating polynomial to $f$ at the nodes $z_k,\,1\le k\le n+1$,
 is denoted by $L_n(z)=L_n\left(z, f(z), R_{n+1}(z)\right)\in\PP_n$.

We also use the notations
$\lesssim$ and $\gtrsim$ in the following sense:
$
\vphi(\tau,\gamma,\ldots)\lesssim \delta(\tau,\gamma,\ldots)$ or
$\vphi(\tau,\gamma,\ldots)\gtrsim \delta(\tau,\gamma,\ldots)$
means that there exists a constant $C>0$  independent of the essential parameters
$\tau,\,\gamma,\ldots$ such that
$\vphi(\tau,\gamma,\ldots)\leq C \delta(\tau,\gamma,\ldots)$ or
$\vphi(\tau,\gamma,\ldots)\geq C \delta(\tau,\gamma,\ldots)$
for the relevant ranges of $\tau,\gamma,\ldots$.
The dependence of $C$ on certain parameters is indicated
by using subscripts, e.g., $\lesssim_{a,b}$. The absence of the subscripts means that
$C$ is an absolute constant.

\subsection{More Preliminaries: Special Sequences of Polynomials}\label{SS1.3}
\setcounter{equation}{0}
Here, we define a special class of sequences of algebraic polynomials
whose zeros  serve as interpolation nodes. Examples of sequences from this
class are discussed as well.
The definition and the examples are taken from author's monograph \cite[Sect. 3]{G2013}.

Let $\be=\be_{(d)}=\left(\be_N\right)_{N=1}^\iy$ and $\g=\g_{(d)}=\left(\g_N\right)_{N=1}^\iy$ be
increasing sequences of positive numbers and $\de=\de_{(d)}=\left(\de_N\right)_{N=1}^\iy$
a decreasing sequence of positive numbers satisfying the conditions
\ba
\lim_{N\to \iy}\be_N=\lim_{N\to \iy}\g_N=\lim_{N\to \iy}\de_N^{-1}=\iy.
\ea

\begin{definition} \label{D1.1}
\cite[Definition 3.1.1]{G2013}
Let $\Pi_d:=\left(P_{2N+d}\right)_{N=1}^\iy$ be a sequence of even (if $d=0$) or odd (if $d=1$)
polynomials $P_{2N+d}(z)=z^{2N+d}+\ldots\in\PP_{2N+d}$,
having only real zeros $\pm x_k=\pm x_{k,N,d},\,1-d\le k\le N,\,x_0=0$,
of multiplicity $1$ and satisfying
the following asymptotic property:
for any $z\in\CC$ with $|z|\le \g_N$, the following inequality
 holds:
 \ba
 \left|\frac{\beta_N^d\,P_{2N+d}(z/\beta_N)}{P_{2N+d}^{(d)}(0)}
 -\cos(z-d\pi/2)\right|
 \lesssim_{\Pi_d}
 \delta_N \min(|z|^2,1)h_d(\vert z\vert).
 \ea
 Then we write $\Pi_d\in\PPP_d(\be,\g,\de)$.
 \end{definition}
 Examples of polynomial sequences, satisfying Definition \ref{D1.1},
 include the Gegenbauer and Hermite polynomials and,
 in addition, the polynomials with equidistant zeros.
 In each example presented below, we include formulae for $P_{2N+d}$ and $\be_N$
  and also conditions on $\g_N$ and $\de_N$ that guarantee the inclusion
  $\left(P_{2N+d}\right)_{N=1}^\iy\in\PPP_d(\be,\g,\de)$.
  Other examples (such as, the Williams--Apostol and Lommel polynomials)
  and a counterexample of the Laguerre polynomials for $\al\ne -1/2$ are given in
  \cite[Sect. 3.2]{G2013}.

 \begin{example}\label{Ex1.2}
 {Normalized Gegenbauer polynomials on $[-1,1]$} (see \cite[Sect. 3.2.1]{G2013}): for $\la\ge 0$,
 \ba
 &&P_{2N+d}=\frac{\G(\la)(2N+d)!}{\G(2N+d+\la)2^{2N+d}}C^\la_{2N+d},\\
 &&\be_N=2N+d+\la, \qquad \g_N=o(N^{2/3}),\qquad \g_N\gtrsim \log(N+1),\qquad \de_N\ge \g_N^3/N^2.
 \ea
 The normalized Legendre polynomials and the normalized Chebyshev polynomials
 $T_{2N+d}$ and $U_{2N+d}$ of the first and second kinds
 are special cases of $P_{2N+d}$ for $\la=1/2,\,0,\,1$, respectively.
 \end{example}
 \begin{example}\label{Ex1.3}
 {Polynomials with equidistant zeros on $[-1,1]$} (see \cite[Sect. 3.2.2]{G2013}):
 \ba
 &&P_{2N+d}(z)=z^d\prod_{k=1}^N
 \left(\left(\frac{2k+d-1}{2N+d-1}\right)^2-z^2\right),\\
 &&\be_N=(2N+d-1)\pi/2, \qquad \g_N=o(N^{1/2}), \qquad \de_N\ge \g_N^2/N.
 \ea
 \end{example}
 \begin{example}\label{Ex1.4}
 {Normalized Hermite polynomials on $\R$} (see \cite[Sect. 3.2.4]{G2013}):
 \ba
 &&P_{2N+d}=2^{-(2N+d)}H_{2N+d},\\
 &&\be_N=\sqrt{4N+2d+1}, \qquad \g_N=o(N^{1/2}),\qquad \g_N\gtrsim \log(N+1),\qquad \de_N\ge \g_N^2/N.
 \ea

 \end{example}

 \section{Polynomial Interpolation of Power Functions}\label{S2}
\setcounter{equation}{0}
Let $h_d$ be defined by \eqref{E1.1}.
The following asymptotic representations for the polynomial interpolation differences
 were proved in \cite[Corollary 4.3.1 (a)]{G2013}.

\begin{theorem}\label{T2.1}
 Let $\left(P_{2N+d}\right)_{N=1}^\iy\in\PPP_d(\be,\g,\de)$
 (see Definition \ref{D1.1})
 and let $\left(y_N\right)_{N=1}^\iy$ be a sequence of real nonzero numbers
 such that $\lim_{N\to\iy} \be_N\left\vert y_N\right \vert =\iy$.
 Then the following statements hold:\\
 (a) If
 $\mathrm{Re}\, s>d$, and $s-d\ne 2,\,4,\ldots,$
 then
 \bna\label{E2.1}
 &&\left|y_N\right|^{s-d}-
 L_{2N}\left(y_N,|y|^{s-d},y^{1-d}P_{2N+d}(y)\right)\nonumber\\
 &&=\frac{2\sin((s-d)\pi/2)}{\pi\beta_N^{s-d}}
     \frac{P_{2N+d}(y_N)}{y_N^{d}P_{2N+d}^{(d)}(0)}
     \left(\int_0^\iy\frac{t^{s-1}}{h_d(t)}dt+o(1)\right),
 \ena
 as $N\to\iy$
 and the convergence in \eqref{E2.1}
 is uniform on the line $\mathrm{Re}\,s=r>d.$\\
 (b) For a fixed $m\in\N,\,1\le m\le N$,
 \bna\label{E2.2}
 &&\left|y_N\right|^{2m}\log\left|y_N\right|
 -L_{2N}\left(y_N,|y|^{2m}\log|y|,y^{1-d}P_{2N+d}(y)\right)\nonumber\\
 &&=\frac{(-1)^m} {\beta_N^{2m}}
 \frac{P_{2N+d}(y_N)}{y_N^{d}P_{2N+d}^{(d)}(0)}
 \left(\int_0^\iy\frac{t^{2m+d-1}}{h_d(t)}dt+o(1)\right),
 \ena
 as $N\to\iy$.
 \end{theorem}

 \begin{remark}\label{R2.2}
 Since $y^{1-d}P_{2N+d}(y)=y\prod_{k=1}^N \left(y^2-x_k^2\right)$,
 the interpolation nodes in \eqref{E2.1} and \eqref{E2.2} are
 $0,\,\pm x_1, \ldots, \pm x_N$.
 \end{remark}

 \begin{remark}\label{R2.3}
 Relations between the integral in \eqref{E2.1} (or \eqref{E2.2})
 and zeta functions are given by the following formula
 (see, e.g., \cite[Sect. 1.12]{ErdI1953}):
 \ba
 \int_0^\iy\frac{t^{s-1}}{h_d(t)}dt=
 \left\{\begin{array}{ll}
 2\Gamma(s)\be(s), &d=0,\\
  2\Gamma(s)\left(1-2^{-s}\right)\z(s), &d=1,\end{array}\right.
 \qquad \mbox{Re}\,s>d.
 \ea
 \end{remark}

 Since the functions $\left|y\right|^{s-d}$ and
 $\left|y\right|^{2m}\log\left|y\right|$ and their
 interpolation polynomials $L_{2N}$ from Theorem \ref{T2.1}
 are even, we can make the substitution $u=y^2$ and discuss
 the corresponding interpolation differences on $[0,\iy)$.
 Let us define
 \beq\label{E2.3}
 Q_N(u):=\prod_{k=1}^N \left(u-x^2_k\right),
 \eeq
where $x_k=x_{k,N,d},\,1\le  k\le N,$ are positive zeros of $P_{2N+d}$.
It is easy to see that
\beq\label{E2.4}
P_{2N+d}^{(d)}(0)=Q_N(0).
 \eeq

In the following corollary, we rewrite relations \eqref{E2.1} and \eqref{E2.2}
 with the interpolation nodes $0,\, x_1^2, \ldots, x_N^2$.

 \begin{corollary}\label{C2.4}
 Let $\left(P_{2N+d}\right)_{N=1}^\iy\in\PPP_d(\be,\g,\de)$
 and let $Q_N$ be defined by \eqref{E2.3},
 where $x_k,\,1\le  k\le N,$ are the positive zeros of $P_{2N+d}$.
 In addition, let $\left(u_N\right)_{N=1}^\iy$ be a sequence of positive numbers
 such that $u_N\notin \{x_1,\ldots,x_N\}$ and
 $\lim_{N\to\iy} \be_N u_N =\iy$.
 Then the following statements hold:\\
 (a) If
 $\mathrm{Re}\, s>d$ and $s-d\ne 2,\,4,\ldots,$
 then
 \bna\label{E2.5}
 \lim_{N\to\iy}\frac{\beta_N^{s-d}Q_N(0)}{Q_N\left(u_N^2\right)}
 \left(u_N^{s-d}-
 L_{N}\left(u^2_N,u^{\frac{s-d}{2}},u Q_N(u)\right)\right)
 =\frac{2\sin((s-d)\pi/2)}{\pi}
     \int_0^\iy\frac{t^{s-1}}{h_d(t)}dt.
 \ena
 The convergence in \eqref{E2.5}
 is uniform on the line $\mathrm{Re\,s}=r>d$.\\
 (b) For a fixed $m\in\N$,
 \bna\label{E2.6}
 \lim_{N\to\iy}\frac{\beta_N^{2m}Q_N(0)}{Q_N\left(u_N^2\right)}
 \left(u_N^{2m}\log u^2_N
 -L_{N}\left(u^2_N,u^{m}\log u,u Q_N(u)\right)\right)
 =(-1)^m
 \int_0^\iy\frac{t^{2m+d-1}}{h_d(t)}dt.
 \ena
 \end{corollary}

 \begin{proof}
 Relations \eqref{E2.5} and \eqref{E2.6} immediately follow from
 \eqref{E2.1}-\eqref{E2.4} if we set
 $u_N=\left\vert y_N\right\vert$.
 The uniform  convergence in \eqref{E2.5} follows from
 Theorem \ref{T2.1} (a).
 \end{proof}

 The proof of Theorem \ref{T5.1} is based on Theorem \ref{T2.1}.
 The following proposition is used in the proof of Theorem \ref{T5.2}.
 Note that the proof of Theorem \ref{T2.1} (a) is based on this proposition as well.

 \begin{proposition}\label{P2.6}
 \cite[Eq. (2.3.21)]{G2013}
 Let $G_{2N}\in\PP_{2N}$ be an even polynomial with $2N$ real zeros of multiplicity $1$,
 and let $\left(y_N\right)_{N=1}^\iy$ be a sequence of real nonzero numbers.
 If
 $0<\mathrm{Re}\, s<2N+1$, and $s\ne 2,\,4,\ldots,$
 then
 \beq\label{E2.10}
 \left|y_N\right|^{s}-
 L_{2N}\left(y_N,|y|^{s},y G_{2N}(y)\right)
 =\frac{2\sin(s\pi/2)}{\pi}G_{2N}(y_N)
     \bigintss_0^\iy\frac{t^{s-1}}{\left(1+\left(t/y_N\right)^2\right)G_{2N}(i t)}dt.
 \eeq
 \end{proposition}

 \section{Properties of $B$-splines and Associated $B$-splines}\label{S3}
\setcounter{equation}{0}

\begin{definition}\label{D3.1}
The $B$-spline $B_N$ of degree $N$ with the set of knots
$\Om=\left\{t_0,\,t_1,\ldots,t_{N+1}\right\}   \subset \R$
is defined by the formulae
\beq\label{E3.1}
B_N(t)=B_N(t,\Om):=(N+1)\sum_{v\in\Om}\frac{(v-t)^N_+}{W^\prime(v)},
\quad t\in\R;
\qquad W(u):=\prod_{v\in \Om}(u-v).
\eeq
\end{definition}

\begin{definition}\label{D3.2}
The associated $B$-spline $B_N^*$ with the set of knots
$\Om=\left\{t_0,\,t_1,\ldots,t_{N+1}\right\}\subset [0,\iy)$
is defined by the formula
\beq\label{E3.2}
B_N^*(t)=B_N^*(t,\Om):=t^{-N}B_N(t,\Om),\qquad t\in\R,
\eeq
where $B_N^*(0):=\lim_{t\to 0+}B_N^*(t)$.
\end{definition}

The next three properties of $B_N$ are well known
(see, e.g., \cite[Theorems 4.17, 4.23, 4.24]{S2007}, respectively).

 \begin{property}\label{P3.3}
 Support and positivity:
 \beq\label{E3.3}
 B_N(t,\Om) \left\{\begin{array}{ll}
 =0, &t\in \R\setminus(\min \Om,\max \Om),\\
 >0, &t\in (\min \Om,\max \Om).\end{array}\right.
 \eeq
 \end{property}
 \begin{property}\label{P3.4}
 Normalization:
 \beq\label{E3.4}
 \int_{\min \Om}^{\max \Om} B_N(t,\Om)dt=1.
 \eeq
 \end{property}
\begin{property}\label{P3.5}
 Symmetry:
 \beq\label{E3.5}
 B_N(t,\Om)=(-1)^{N+1}(N+1)\sum_{v\in\Om}\frac{(t-v)^N_+}{W^\prime(v)},
 \eeq
 where $W$ is defined in \eqref{E3.1}.
 \end{property}
 It is easy to prove the next two properties of $B_N$ and $B_N^*$.

\begin{property}\label{P3.6}
First interval:
 \beq\label{E3.6}
 B_N(t,\Om)=\frac{(-1)^{N+1}(N+1)(t-\min \Om)^N}{W^\prime(\min \Om)},
 \qquad t\in[\min \Om,\min (\Om\setminus \{\min \Om\})].
 \eeq
 \end{property}
 \begin{proof}
 The property immediately follows from \eqref{E3.5}.
 \end{proof}
 \begin{property}\label{P3.8}
Support and positivity of $B_N^*$:
\beq\label{E3.9}
 B_N^*(t,\Om) \left\{\begin{array}{lll}
 =0, &t\in (0,\min \Om]\cup [\max \Om,\iy),\\
 =\frac{(-1)^{N+1}(N+1)}{W^\prime(\min \Om)}, &t=\min \Om=0,\\
 >0, &t\in (\min \Om,\max \Om).\end{array}\right.
 \eeq
 \end{property}
\begin{proof}
Relation \eqref{E3.9} immediately
follows from relations \eqref{E3.2}, \eqref{E3.3}, and \eqref{E3.6}.
\end{proof}
\noindent
Note that relation \eqref{E3.9} shows that $B_N^*$ has a jump
discontinuity at $\min \Om$ if $\min \Om=0$.

Finally, we prove an important extension of
the Curry--Schoenberg result \cite[p. 74]{CS1966}.
 \begin{property}\label{P3.7}
Polynomial interpolation:\\
If $f\in C ([\min \Om,\max \Om])\cap C^{(N+1)}
((\min \Om,\max \Om])$ and
\beq\label{E3.7}
\int_{\min \Om}^{\max \Om} \left\vert f^{(N+1)}(t)\right
\vert(t-\min \Om)^Ndt<\iy,
\eeq
then
 \beq\label{E3.8}
 f(u)-L_N(u,f(u),w(u))
 =\frac{w(u)}{(N+1)!}\int_{\min \Om}^{\max \Om}
 B_N(t,\Om)f^{(N+1)}(t)\,dt,
 \qquad u\in\Om,
 \eeq
 where $w(u):=\prod_{v\in\Om,v\ne u}(u-v)$.
\end{property}
\begin{proof}
Assume, without loss of generality, that $\min \Om=0$ and $\max \Om=1$.
Next, let $\tilde{f}$ be an extension  of $f$ to $[0,2]$ such
that $\tilde{f}\in C^{(N+1)}((0,2])$.

Then Property \ref{P3.7} is a consequence of the following property
of the weighted integral modulus of continuity:
\beq\label{E3.8a}
\lim_{\tau\to 0+}\int_0^1\left\vert F(t+\tau)-F(t)\right\vert\,t^N\,dt=0,
\eeq
where $F:=\tilde{f}^{(N+1)}$.
Indeed, if $f\in C^{(N+1)} ([0,1])$, then \eqref{E3.8} is
well known (see, e.g., \cite[p. 74]{CS1966} or \cite[Theorem 4.23]{S2007}).
Hence the following equality holds:
\beq\label{E3.8c}
 \tilde{f}(u+\tau)-L_N(u,\tilde{f}(u+\tau),w(u))
 =\frac{w(u+\tau)}{(N+1)!}\int_{0}^{1}
 B_N(t,\Om)F(t+\tau)\,dt,
 \qquad u\in\Om,
 \eeq
since $\tilde{f}(\cdot+\tau)\in C^{(N+1)}([0,1])$ for any $\tau\in (0,1]$.
Then setting $\tau\to 0+$ in \eqref{E3.8c}, we arrive at the left-hand side of \eqref{E3.8}
since $f\in C([0,1])$,
while the right-hand side of \eqref{E3.8} follows from
\eqref{E3.8c} and \eqref{E3.8a}
since $ B_N^*(t,\Om)$ is bounded on $[0,1]$
by Definition \ref{D3.2} and Property \ref{P3.8}.

To prove \eqref{E3.8a}, we note that for any $\tau$ and $\de$,
satisfying $0<\tau\le \de\le 1/2$, the following inequalities are valid:
\bna\label{E3.8b}
&&\int_0^1\left\vert F(t+\tau)-F(t)\right\vert\,t^N\,dt\nonumber\\
&&\le \int_0^\de\left\vert F(t)\right\vert\,t^N\,dt
+\int_0^\de\left\vert F(t+\tau)\right\vert\,t^N\,dt
+\int_\de^1\left\vert F(t+\tau)-F(t)\right\vert\,t^N\,dt\nonumber\\
&&\le 2\int_0^{2\de}\left\vert F(t)\right\vert\,t^N\,dt
+\int_\de^1\left\vert F(t+\tau)-F(t)\right\vert\,dt
:=I_1+I_2.
\ena
Furthermore, in view of condition \eqref{E3.7}, we see that given $\vep>0$,
there exists
$\de=\de(\vep)\in(0,1/2]$
such that $I_1<\vep/2$. Finally,  there exists
$\tau=\tau(\vep)\in(0,\de]$
such that $I_2<\vep/2$. Thus, $I_1+I_2< \vep$ and \eqref{E3.8a} follows from
\eqref{E3.8b}. This completes the proof of the property.
\end{proof}

Next, we use Properties \ref{P3.8} and \ref{P3.7} to obtain one more
representation for the polynomial interpolation difference.

\begin{corollary}\label{C3.9}
Let $\min \Om=0,\,u\in\Om$, and $w(u)=\prod_{v\in\Om,v\ne u}(u-v)$.
Then the following statements hold:\\
(a) If $\mathrm{Re}\, s>d$, and $s-d\ne 2,\,4,\ldots$, then
\beq\label{E3.10}
u^{\frac{s-d}{2}}-L_{N}\left(u,u^{\frac{s-d}{2}},w(u)\right)
=M_{N,s,d}\, w(u) \int_{0}^{\max \Om} B_N^*(t,\Om) t^{\frac{s-d}{2}-1}dt,
\eeq
where
\beq\label{E3.11}
M_{N,s,d}:=(-1)^{N+1}\frac{\G(N+1-(s-d)/2)}{\G((d-s)/2)(N+1)!}.
\eeq
(b) If $m\in\N$ and $m<N+1$,  then
\beq\label{E3.12}
(1/2)u^m\log u-L_N\left(u,(1/2)u^m\log u, w(u)\right)
=\Tilde{M}_{N,m}\, w(u)\int_{0}^{\max \Om} B_N^*(t,\Om) t^{m-1}dt,
\eeq
where
\beq\label{E3.13}
\Tilde{M}_{N,m}:=(1/2)(-1)^{m+N}\frac{m!(N-m)!}{(N+1)!}.
\eeq
\end{corollary}
\begin{proof}
(a) Setting $f(u):=u^{\frac{s-d}{2}}$, we see that
\beq\label{E3.14}
f^{(N+1)}(t)=(N+1)!M_{N,s,d}\,t^{\frac{s-d}{2}-N-1},\qquad t>0,
\eeq
where $M_{N,s,d}$ is defined by \eqref{E3.11}.
Hence $f$ satisfies all the conditions
of Property \ref{P3.7} on $[0,\max\Om]$, and \eqref{E3.10}
 follows from
\eqref{E3.8}, \eqref{E3.14}, and Definition \ref{D3.2}.\\
(b)  Setting $f(u):=(1/2)u^m\log u$, we obtain from
the general Leibniz rule that
for $1\le m< N+1$,
\beq\label{E3.15}
f^{(N+1)}(t)=(-1)^N(1/2)(N+1)!t^{m-N-1}
\sum_{k=0}^m\frac{(-1)^k\binom{m}{k}}{N+1-k}
=(N+1)!\Tilde{M}_{N,m}t^{m-N-1},
\eeq
where $\Tilde{M}_{N,m}$ is defined by \eqref{E3.13} and
 the last equality in \eqref{E3.15} follows from the
partial fraction decomposition
\ba
\frac{1}{\prod_{k=0}^m(z-k)}=\frac{(-1)^m}{m!}
\sum_{k=0}^m\frac{(-1)^k\binom{m}{k}}{z-k}
\ea
for $z=N+1$. Hence $f$ satisfies all the conditions
of Property \ref{P3.7} on $[0,\max\Om]$, and \eqref{E3.12}
follows from
\eqref{E3.8}, \eqref{E3.15}, and Definition \ref{D3.2}.
\end{proof}

The following corollary is a combination of
Corollaries \ref{C2.4} and \ref{C3.9}.

\begin{corollary}\label{C3.10}
Let $x_k=x_{k,N,d},\,1\le  k\le N,$ be positive zeros of $P_{2N+d}$, where
$\left(P_{2N+d}\right)_{N=1}^\iy\in\PPP_d(\be,\g,\de)$.
 In addition, let $\left(u_N\right)_{N=1}^\iy$ be a sequence of positive numbers
 such that $u_N\notin \{x_1,\ldots,x_N\}$ and
 $\lim_{N\to\iy} \be_N u_N =\iy$.
 If $\mathrm{Re} \,s>d$, then for the set of knots
 $\Om_{N+2,d}:=\{0, u_N^2, x_{1,N,d}^2,\ldots, x_{N,N,d}^2\}$
  the following relation holds:
 \bna\label{E3.16}
 &&\G\left(1+\frac{s-d}{2}\right)
 \lim_{N\to\iy}\left[u_N^2 \,\be_N^{s-d}\,N^{-1-(s-d)/2}\,\prod_{k=1}^Nx_{k,N,d}^2\right.\nonumber\\
 &&\times\left.\int_0^{\max \Om_{N+2,d}}B_N^*\left(t,\Om_{N+2,d}\right)t^{(s-d)/2-1}dt\right]
 =\int_0^\iy\frac{t^{s/2-1}}{h_d(\sqrt{t})}dt.
 \ena
 The convergence in \eqref{E3.16} is uniform on any compact subset
  of the line $\mathrm{Re}\,s=r>d$.
\end{corollary}
\begin{proof}
The proof is based on Corollaries \ref{C2.4} and \ref{C3.9}.
To prove \eqref{E3.16} for $\mbox{Re} \,s>d$ and $s-d\ne 2,\,4,\ldots$,
we first use statement (a) of Corollary \ref{C2.4} and
replace the interpolation difference in the left-hand side of \eqref{E2.5}
by the right-hand side of \eqref{E3.10} with $w(u)=u\,Q_N(u)$
and $u=u_N^2$
(recall that $Q_N$ is defined by \eqref{E2.3}).

Next, we simplify the limit expression by using  Euler's reflection formula
\beq\label{E3.18}
\frac{2}{\pi}\G\left(\frac{d-s}{2}\right)\sin \left(\frac{(s-d)\pi}{2}\right)
=-\frac{2}{\G\left(1+{(s-d)}/{2}\right)}
\eeq
and the limit relation (see, e.g., \cite[Eq. 1.18(5)]{ErdI1953})
\beq\label{E3.19}
\lim_{N\to\iy}\frac{\G(N+1-(s-d)/2)}{(N+1)!}\,N^{1+(s-d)/2}=1.
\eeq
Thus, we arrive at \eqref{E3.16} for $\mbox{Re} \,s>d$ and $s-d\ne 2,\,4,\ldots$.
The uniform convergence in \eqref{E3.16}
  follows from Corollary \ref{C2.4} since the
convergence in \eqref{E3.19} is uniform on any compact subset
  of the line $\mathrm{Re}\,s=r>d$ (see \cite[Eq. 1.18(4)]{ErdI1953}).

It turns out that \eqref{E3.16} holds for $s-d=2m,\,m\in\N$, as well.
Indeed, using statement (b) of Corollary \ref{C2.4} and then
replacing the interpolation difference in the left-hand side of \eqref{E2.6}
by the right-hand side of \eqref{E3.12} with $w(u)=u\,Q_N(u)$ and $u=u_N^2$,
we arrive at \eqref{E3.16} for $s-d=2m,\,m\in\N$.
\end{proof}
\begin{remark}\label{R4.11}
The normalization property for $B_N(\cdot,\Om)$ is given by \eqref{E3.4}.
Setting $s=2+d$ in \eqref{E3.16}, we obtain the
following limit version of
the normalization property for $B_N^*\left(\cdot,\Om_{N+2,d}\right)$:
\ba
\lim_{N\to\iy}u_N^2 \,\be_N^{2}\,N^{-2}\,\prod_{k=1}^Nx_{k,d}^2\,
 \int_0^{\max \Om_{N+2,d}}B_N^*\left(t,\Om_{N+2,d}\right)dt
 =\left\{\begin{array}{ll}
 4\be(2),&d=0,\\
 14\z(3),&d=1. \end{array}\right.
\ea
Moreover, setting $s=2m+2+d$ in \eqref{E3.16},
we obtain asymptotic formulae for the moments of
$B_N^*\left(\cdot,\Om_{N+2,d}\right)$.
Note that combinatorial representations for the moments of $B$-splines were obtained
 by Neuman  \cite[Sect. 3]{N1981}.
\end{remark}

\section{Theta-like Functions and their Properties}\label{S4}
\setcounter{equation}{0}
Let us define two theta-like functions,
\beq\label{E4.1}
\varTheta_d(t):=\left\{\begin{array}{ll}
1-\frac{4}{\pi}
\sum_{k=0}^\iy(-1)^k\frac{\exp\left(-(\pi^2/4)(2k+1)^2t^{-1}\right)}{2k+1}, &d=0,\\
1+2\sum_{k=1}^\iy(-1)^k\,\exp\left(-(\pi k)^2\,t^{-1}\right), &d=1,
\end{array}\right. \qquad t>0,
\eeq
where $\varTheta_1$ coincides with $\vartheta_4\left(0,i\pi t^{-1}\right)$,
while $\varTheta_0$ is related to derivatives and integrals of $\vartheta_1\left(0,i\pi t^{-1}\right)$
(see \cite[p. 548]{B1998} and \eqref{E4.5}). Note that the theta functions
$\vartheta_4$ and $\vartheta_1$ are given in the notation of Rademacher \cite[p. 166]{R1973}.

Certain properties of $\varTheta_d$ are discussed below.

\begin{lemma}\label{L4.1}
(a) $\varTheta_d$ is continuous on $(0,\iy)$
and $\sup_{t\in (0,\iy)}\left\vert \varTheta_d(t)\right\vert<\iy$.\\
(b) The following equality holds:
\beq\label{E4.3}
H_d(t):=t\int_0^\iy \varTheta_d(1/z)e^{-tz}dz
=\frac{t^{d/2}}{h_d\left(\sqrt{t}\right)}, \qquad t>0,
\eeq
where $h_d$ is defined by \eqref{E1.1}.\\
(c) The following asymptotics hold:
\beq\label{E4.2}
\varTheta_d(t)=\left\{\begin{array}{ll}
4(\pi t)^{-1/2}e^{-t/4}\left(1+O\left(t^{-1}\right)\right), &d=0,\\
2(t/\pi)^{1/2}e^{-t/4}\left(1+O\left(e^{-2t}\right)\right), &d=1,
\end{array}\right. \qquad t\to\iy.
\eeq
\end{lemma}
\begin{proof}
(a) The statement immediately follows from \eqref{E4.1}. \\
(b) We note that the Laplace transformation $H_d(t)/t$ exists for $t>0$ by statement (a).
Then integrating the corresponding exponential functions from \eqref{E4.3} and \eqref{E4.1}
and using the partial fraction decomposition of hyperbolic functions, we have for $t>0$,
\ba
&&H_0(t)=1-\frac{4}{\pi}\sum_{k=0}^\iy(-1)^k \frac{1}{(2k+1)\left(\pi^2(k+1/2)^2+t\right)}
=\frac{4}{\pi}\sum_{k=0}^\iy(-1)^k \frac{2k+1}{\pi^2(k+1/2)^2+t}
=\frac{1}{\cosh \sqrt{t}};\\
&&H_1(t)=1+2t\sum_{k=0}^\iy(-1)^k \frac{1}{(\pi k)^2+t}=\frac{\sqrt{t}}{\sinh \sqrt{t}}.
\ea
Hence \eqref{E4.3} is valid.\\
(c) The following
identity holds by properties of theta functions:
\ba
\varTheta_1(t)=2(t/\pi)^{1/2}e^{-t/4}\left(1+\sum_{k=1}^\iy e^{-k(k+1)t}\right),
\qquad t>0,
\ea
see \cite[p. 177]{R1973} or \cite[p. 548]{B1998}. Hence \eqref{E4.2}
is valid for $d=1$.

To prove \eqref{E4.2} for $d=0$, we first need the following relations:
\bna\label{E4.4}
\varTheta_0^*(x)
&:=&\sum_{k=0}^\iy(-1)^k (2k+1)\exp\left(-(\pi^2/4)(2k+1)^2x\right)\nonumber\\
&=&(\pi x)^{-3/2}e^{-1/(4x)}\left(1+\sum_{k=1}^\iy (2k+1)e^{-k(k+1)x^{-1}}\right)\nonumber\\
&=&(\pi x)^{-3/2}e^{-1/(4x)}\left(1+O\left(e^{-2x^{-1}}\right)\right),
\qquad x\to 0+,
\ena
see \cite[pp. 548--549]{B1998}. Next, it follows from \eqref{E4.1} and \eqref{E4.4} that
\bna\label{E4.5}
\varTheta_0(t)
&=&\pi\int_0^{1/t}\varTheta_0^*(x)dx\nonumber\\
&=&\pi^{-1/2}\int_0^{1/t}x^{-3/2}e^{-1/(4x)}\left(1+O\left(e^{-2/x}\right)\right)dx\nonumber\\
&=&2\pi^{-1/2}\G(1/2,t/4)\left(1+O\left(e^{-2t}\right)\right),\qquad t\to \iy.
\ena
Finally, using  the asymptotic behavior of $\G(1/2,\tau)$ as $\tau\to \iy$
(see \cite[Sect. 9.2]{ErdII1953}), we obtain \eqref{E4.2} for $d=0$ from \eqref{E4.5}.
\end{proof}

\section{The Mellin Transform}\label{S6}
\setcounter{equation}{0}
We first need certain known facts
 about the \textit{Mellin transform} defined by the formula
\ba
M(F,s):=\int_0^\iy F(t)\,t^{s-1}dt,\qquad \mbox{Re}\,s>0,\quad t^{\mathrm{Re}\,s-1} F(t)\in L_1([0,\iy)).
\ea
 In particular, $M(F,\cdot)$ exists for $F\in\LL$.

 In addition,
 the \textit{Mellin convolution}
 of two functions $F\in\LL$ and $G\in\LL$ is defined by the integrals
\beq\label{E1.3}
H(t):=\int_0^\iy F(\tau)G\left(\frac{t}{\tau}\right)\,\frac{d\tau}{\tau}
=\int_0^\iy G(\tau)F\left(\frac{t}{\tau}\right)\,\frac{d\tau}{\tau}.
\eeq

\begin{proposition}\label{P1.1a} (see, e.g., \cite[Sect. 4, Theorem 3]{BJ1997})
The Mellin convolution $H$ defined by \eqref{E1.3}
belongs to $\LL$. Moreover, the following relation holds:
\beq\label{E1.4}
M(H,s)=M(F,s)M(G,s).
\eeq
\end{proposition}
\begin{proposition}\label{P1.1b} (see, e.g., \cite[Sect. 6, Theorem 6]{BJ1997})
If $F\in\LL$ and $M(F,r+i\cdot)\in L_1(\R)$ for some $r\in\R$, then
\ba
F(t)=\frac{1}{2\pi}\int_{-\iy}^\iy M(F,r+i\tau)\,t^{-r-i\tau}d\tau,
\qquad t>0.
\ea
\end{proposition}

Next, we define a condition in terms of the Mellin transform
 that is used in Theorem \ref{T5.2}.

\begin{definition}\label{D6.1}
Let $\left(u_N\right)_{N=1}^\iy$ be a sequence of positive numbers,
and let $\be=\left(\be_N\right)_{N=1}^\iy$ be an increasing sequence of positive numbers
with $\lim_{N\to\iy}\be_N=\iy$.
In addition, let $\left(G_{2N}\right)_{N=1}^\iy$ be a sequence of even
polynomials from $\PP_{2N}$,
having $2N$ real zeros
of multiplicity $1,\,G_{2N}(0)\ne 0$. We say that $\left(G_{2N}\right)_{N=1}^\iy$ satisfies the
\textit{$(r,\be)$-Condition} for a given $r> 0$ if there exist constants $N_0=N_0(r)$
and $\mu=\mu(r)$ such that the following inequality holds:

\beq\label{E6.1}
\sup_{N\in\N,\,N\ge N_0} \be_N^{r}
\left\vert\bigintss_0^\iy\frac{t^{s-1}}{\left(1+\left(t/u_N\right)^2\right)G_{2N}(i t)}dt\right\vert
\lesssim_r \left\vert\mathrm{Im}\,s\right\vert^\mu e^{-(\pi/2)\left\vert\mathrm{Im}\,s\right\vert},
\qquad\mathrm{Re}\,s=r.
\eeq
\end{definition}

The following estimates of the Fourier and Mellin transforms of  analytic functions
are helpful  for the verification  of the $(r,\be)$-Condition for certain polynomials.
The Fourier transform of a function $f\in L_1(\R)$
   is denoted by the formula
   \ba
   \widehat{f}(v):=
   \int_{\R}f(x)e^{i v x}dx,\qquad v\in\R.
   \ea

\begin{lemma}\label{L6.2}
Let a complex-valued function $f$ satisfy the following conditions:
\begin{itemize}
\item[(i)] $f$ is analytic  on the strip $S_\de:=\{x+iy: x\in\R, \,\vert y\vert< \de\},\,\de>0$;
\item[(ii)] $f$ is real-valued on $\R$;
 \item[(iii)]
 $\sup_{\xi\in S_\de} \vert \mathrm{Re}\,f(\xi) \vert<\iy$;
 \item[(iv)]
 $\sup_{\vert y\vert< \de}\int_{\R} \vert \mathrm{Re}\,f(x+iy) \vert\,dx \le K(\de)<\iy$.
\end{itemize}
Then $f\in L_1(\R)$ and
\beq\label{E6.2}
\left\vert\widehat{f}(v)\right\vert \le 2 K(\de) e^{-\de \vert v\vert},\qquad v\in\R.
\eeq
\end{lemma}
\begin{proof}
If conditions (i), (ii), and (iii) are satisfied, then it is known
(see \cite[Sect. 110]{A1965} or \cite[Sect. 3.8.5]{T1963}) that for almost all $x\in \R$
there exists the limit
\beq\label{E6.3}
\lim_{y\to \pm\de} \mathrm{Re}\,f(x+iy)=g(x),\qquad \sup_{x\in\R}\vert g(x)\vert<\iy,
\eeq
and the following representation holds:
\beq\label{E6.4}
f(x)=\frac{1}{2\de}\bigintss_{\R} \frac{g(t)}{\cosh\frac{\pi(x-t)}{2\de}}dt.
\eeq
In addition to \eqref{E6.3}, condition (iv) implies that $g\in L_1(\R)$
since by \eqref{E6.3}, Fatou's Lemma, and condition (iv),
\beq\label{E6.5}
\int_\R\vert g(x)\vert dx
\le \liminf_{y\to \pm\de} \int_\R\vert \mathrm{Re}\,f(x+iy)\vert dx
\le K(\de).
\eeq
Then \eqref{E6.2} follows from \eqref{E6.4} and \eqref{E6.5} since
\ba
\left\vert\widehat{f}(v)\right\vert
=\left\vert\widehat{g}(v)\right\vert/\cosh \de v
\le 2 K(\de) e^{-\de \vert v\vert}.
\ea
\end{proof}
\begin{lemma}\label{L6.3}
Let $r>0$ be a fixed number, and
let a complex-valued function $F$ satisfy the following conditions:
\begin{itemize}
\item[(i)] $F$ is analytic  on the region
$A_\de:=\{w\in\CC: \,\vert \arg w\vert< \de\},\,0<\de< \pi$;
\item[(ii)] $F$ is real-valued on $\R$;
 \item[(iii)]
 $\sup_{w\in A_\de} \left\vert \mathrm{Re}\, \left(w^r F(w)\right)\right\vert<\iy$;
 \item[(iv)]
 $\sup_{\vert \arg w\vert< \de}
 \int_{0}^\iy R^{r-1}\left\vert \mathrm{Re}\,
 \left(e^{ir\arg w}F\left(R\, e^{i\arg w}\right)\right) \right\vert\,dR \le K_1(\de,r)<\iy$.
\end{itemize}
Then
\ba
\left\vert M(F,r+iv)\right\vert \le 2 K_1(\de,r) e^{-\de \vert v\vert},\qquad v\in\R.
\ea
\end{lemma}
\begin{proof}
Making the standard substitution $t=e^x$ in the integral
$\int_0^\iy F(t)\,t^{r+iv-1}dt$, we reduce Lemma \ref{L6.3} to Lemma \ref{L6.2}
for $f(\xi)=e^{r\xi} F(e^\xi)$.
\end{proof}

\begin{remark}\label{R6.4}
Different versions of Lemmas \ref{L6.2} and \ref{L6.3} are well known
(see, e.g., \cite[Theorems 26 and 31]{T1937} and \cite[Sect. 6]{E1986}).
\end{remark}

In the verification of the $(r,\be)$-Condition, the following version of Lemma \ref{L6.3}
is helpful.
\begin{lemma}\label{L6.5}
Let $r>0$ be a fixed number and
let $u_N,\,\be_N$, and $G_{2N},\,N>r/2$, be numbers and a
polynomial from Definition \ref{D6.1}.
If $\al\in(0,\pi/2)$ and
there exists $N_0=N_0(r)$ such that
\beq\label{E6.7}
\sup_{N\in\N,\,N\ge N_0}\be_N^r\sup_{\vert \arg w\vert< \pi/2-\al}
 \int_{0}^\iy R^{r-1}\left\vert
 G_{2N}\left(iR\, e^{i\arg w}\right) \right\vert^{-1}\,d R \le K_2(\al,r)<\iy,
\eeq
then for the function
$F_N(t):=\left(\left(1+\left(t/u_N\right)^2\right)G_{2N}(i t)\right)^{-1}$,
\beq\label{E6.8}
\sup_{N\in\N,\,N\ge N_0}\be_N^r\left\vert M\left(F_N,r+iv\right)\right\vert
\le (\pi/\al) K_2(\al,r) e^{- (\pi/2-\al)\vert v\vert},\qquad v\in\R.
\eeq
\end{lemma}
\begin{proof}
Note first that $\left(1+\left(\cdot/u_N\right)^2\right)G_{2N}(i\cdot)$
is a real-valued polynomial on $\R$ with all imaginary zeros.
Then it is easy to see that conditions (i), (ii), and (iii) of Lemma \ref{L6.3}
with $\de=\pi/2-\al$
 are satisfied for the function $F=F_N$
from Lemma \ref{L6.5}.
Next, for $\al\in (0,\pi/2)$,
\ba
\inf_{\vert t\vert\le \pi/2-\al}\inf_{\la\ge 0}\left\vert 1+\la e^{2it}\right\vert
\ge \min\{1,(4/\pi)\al\}.
\ea
Hence
$\left\vert 1+\left(R/u_N\right)^2 e^{2it}\right\vert
\ge (2/\pi)\al$ for $R\ge 0$ and $N\in\N$.
Therefore, if \eqref{E6.7} from Lemma \ref{L6.5} is valid, then
condition (iv) from Lemma \ref{L6.3} with $\de=\pi/2-\al$ and
 $K_1(\de,r)=(\pi/2)\al^{-1}K_2(\de,r)$ is satisfied as well
for the function $F=F_N$ from Lemma \ref{L6.5}.
Thus, Lemma \ref{L6.5} follows from Lemma \ref{L6.3}.
\end{proof}

In the next lemma, we show that sequences of Chebyshev-like polynomials satisfy
the $(r,\be)$-Condition.

\begin{lemma}\label{L6.6}
Let $G_{2N}(t)=t^{-d}Q_{2N+d,\la}(t)/Q_{2N+d,\la}^{(d)}(0)$ and $\be_N=\be_{N,\la}:=2N+d+\la$,
where $\la=0$ or $\la=1$ and
\bna
&&Q_{2N+d,0}(t):=T_{2N+d}(t):=(1/2)\left(\left(t+\sqrt{t^2-1}\right)^{2N+d}
+\left(t+\sqrt{t^2-1}\right)^{-(2N+d)}\right),\label{E6.9}\\
&&Q_{2N+d,1}(t):=U_{2N+d}(t):=
\frac{\left(t+\sqrt{t^2-1}\right)^{2N+d+1}
-\left(t+\sqrt{t^2-1}\right)^{-(2N+d+1)}}{2\sqrt{t^2-1}},\label{E6.9a}\\
&&T_{2N+d}^{(d)}(0)=(-1)^N(2N+1)^d, \qquad
U_{2N+d}^{(d)}(0)=(-1)^N(2N+2)^d.
\label{E6.10}
\ena
Then inequality \eqref{E6.1} holds for any $r>1-d,\,\mu=3+\max\{r+d-2,0\}$,
and $N_0=N_0(r,d,\la)$.
\end{lemma}
\begin{proof}
We first prove the following inequality:
\bna\label{E6.11}
&&\sup_{N\in\N,\,N\ge N_0}\be_{N,\la}^r Q_{2N+d,\la}^{(d)}(0) \sup_{\al<\arg w< \pi-\al}
 \int_{0}^\iy R^{r+d-1}\left\vert
 Q_{2N+d,\la}\left(R\, e^{i\arg w}\right) \right\vert^{-1}\,d R\nonumber\\
 &&\lesssim_{r,d,\la} \al^{-(2+\max\{r+d-2,0\})},
 \ena
 where $\al\in(0,\pi/2)$ is a fixed number and $N_0=N_0(r,d,\la)$.

 It follows from \eqref{E6.9} and \eqref{E6.9a} that
 \bna
 &&T_{2N+d}(w)=(1/2)\left(z^{2N+d}+z^{-(2N+d)}\right),\qquad w=R e^{i\arg w},\quad z=\rho e^{it},\label{E6.12}\\
 &&U_{2N+d}(w)=\frac{
 z^{2N+d+2}-z^{-(2N+d)}}{z^2-1},\qquad w=R e^{i\arg w}
 ,\quad z=\rho e^{it},\label{E6.12a}
 \ena
 where $z$ is a point outside of the unit circle
 and $w=(1/2)(z+1/z)$ is a point on the ellipse with the foci $(-1,0)$ and $(1,0)$
 and with the sum of its semi-axes equal to $\rho>1$. Since
 $
 w=(1/2)((\rho+1/\rho)\cos t+i(\rho-1/\rho) \sin t),
 $
 we see that
 \bna
 &&R=R(\rho,\arg w)=\frac{\rho^2-1}{2}\sqrt{\frac{1}{\rho^2}+\frac{4}
 {(\rho^2-1)^2+(\rho^2+1)^2\tan^2(\arg w)}},\label{E6.13}\\
 &&\frac{d R(\rho,\arg w)}{d \rho}=\frac{\rho(\rho^4-1)}{4R(\rho,\arg w)}
 \sqrt{\frac{1}{\rho^4}+\frac{16 \tan^2(\arg w)}
 {\left((\rho^2-1)^2+(\rho^2+1)^2\tan^2(\arg w)\right)^2}}.\label{E6.14}
 \ena
 Let $w$ satisfy the condition $\al<\arg w< \pi-\al$,
 where $\al\in(0,\pi/2)$ is a fixed number (cf. \eqref{E6.11}).
 Then $\al^2\le \tan^2(\arg w)$ and
 the following relations are immediate consequences of \eqref{E6.13} and \eqref{E6.14}:
 \beq\label{E6.15}
 \frac{\left(\rho^2-1\right)}{\rho}
 \lesssim R(\rho,\arg w) \lesssim \frac{\left(\rho^2-1\right)}{\al},\qquad
 0\le\frac{d R(\rho,\arg w)}{d \rho} \lesssim \frac{\rho\left(\rho^4-1\right)}{\al^2 \,R(\rho,\arg w)}.
 \eeq
 Using substitution \eqref{E6.13}, relations \eqref{E6.12}, \eqref{E6.12a},  \eqref{E6.15},
 and the elementary inequality
 \ba
 \frac{\rho-1}{\rho^{n+1}-1}\le \frac{1}{(n-k+1)\rho^k},\qquad \rho> 1, \quad k,n\in\N,\quad 1\le k\le n,
 \ea
 we obtain for $N>N_0(r,d,\la)$,
 \ba
 &&\int_{0}^\iy R^{r+d-1}\left\vert
 Q_{2N+d,\la}\left(R\, e^{i\arg w}\right) \right\vert^{-1}\,d R
 \lesssim_\la \frac{1}{\al^2}\int_{1}^\iy \frac{(R(\rho,\arg w))^{r+d-2}\,\rho\left(\rho^4-1\right)}
 {\rho^{2N+d+2\la}-\rho^{-(2N+d)}}d\rho\\
 &&\lesssim_{r,d,\la} \frac{1}{N\al^{2+\max\{r+d-2,0\}}}\int_{1}^\iy \frac{(\rho-1)^{r+d-2}}{\rho^N}d\rho
 \lesssim_{r,d,\la} \frac{1}{\al^{2+\max\{r+d-2,0\}}N^{r+d}}.
 \ea
 Hence we arrive at \eqref{E6.11}.
 Therefore, the constant $K_2$ from \eqref{E6.7} satisfies the inequality
 \beq\label{E6.15a}
 K_2(\al,r)\lesssim_{r,d,\la} \al^{-(2+\max\{r+d-2,0\})}.
 \eeq
Next, using Lemma \ref{L6.5}, we have from \eqref{E6.8}, \eqref{E6.10}, and \eqref{E6.15a}
that for any $\al\in (0,\pi/2)$,
the following inequality holds:
 \beq\label{E6.16}
\sup_{N\in\N,\,N\ge N_0} \be_{N,\la}^{r}
\left\vert\bigintss_0^\iy\frac{t^{s+d-1}Q_{2N+d,\la}^{(d)}(0)}{\left(1+\left(t/u_N\right)^2\right)
Q_{2N+d,\la}(i t)}dt\right\vert
\lesssim_{r,d,\la} \frac{1}{\al^{3+\max\{r+d-2,0\}}}
e^{-(\pi/2-\al)\left\vert\mathrm{Im}\,s\right\vert}.
\eeq
Finally, choosing
\ba
\al=\left\{\begin{array}{ll}
\left\vert\mathrm{Im}\,s\right\vert^{-1}, &\left\vert\mathrm{Im}\,s\right\vert>3/\pi,\\
\pi/3, &\left\vert\mathrm{Im}\,s\right\vert\le 3/\pi
\end{array}\right.
\ea
in \eqref{E6.16}, we arrive at \eqref{E6.1} with $\mu=3+\max\{r+d-2,0\}$.
\end{proof}

\section{Asymptotics of $B_N^*$}\label{S5}
\setcounter{equation}{0}
The limit relations between the Mellin transforms of the scaled associated
$B$-spline $B_N^*$ (see Definition \ref{D3.2}) and
the theta-like functions $\varTheta_d$ (see \eqref{E4.1}) are discussed in the following theorem.

\begin{theorem}\label{T5.1}
Let $x_k=x_{k,N,d},\,1\le  k\le N,$ be positive zeros of $P_{2N+d}$, where
$\left(P_{2N+d}\right)_{N=1}^\iy\in\PPP_d(\be,\g,\de)$
(see Definition \ref{D1.1}).
 In addition, let $\left(u_N\right)_{N=1}^\iy$ be a sequence of positive numbers
 such that $u_N\notin \{x_1,\ldots,x_N\}$ and
 $\lim_{N\to\iy} \be_N u_N =\iy$.
 If $\mathrm{Re} \,s>d$, then for the set of knots
 $\Om_{N+2,d}:=\{0, u_N^2, x_{1,N,d}^2,\ldots, x_{N,N,d}^2\}$
  the following relation holds:
 \bna\label{E5.1}
 \lim_{N\to\iy}\left[\frac{u_N^2 \,\prod_{k=1}^Nx_{k,N,d}^2}{N}\,
 \int_0^{\iy}B_N^*\left(\frac{N}{\be^2_N}t,\Om_{N+2,d}\right)t^{(s-d)/2-1}dt\right]
 =\int_0^\iy \varTheta_d(t) t^{(s-d)/2-1}dt.
 \ena
 The convergence in \eqref{E5.1} is uniform on any compact subset
  of the line $\mathrm{Re}\,s=r>d$.
\end{theorem}
\begin{proof}
 Note first that $F(\tau):=\varTheta_d(\tau)$ belongs to $\LL$ by Lemma \ref{L4.1} (b)
  and also $G(\tau):=\tau e^{-\tau}\in \LL$.
  Then by relation \eqref{E4.3},
 the Mellin convolution $H_d$ of two functions $F$ and $G$
 (see \eqref{E1.3}) is
 \beq\label{E5.3}
 H_d(t)=\int_0^\iy \varTheta_d(t/\tau)e^{-\tau}d\tau
 =t\int_0^\iy \varTheta_d(1/z)e^{-tz}dz
=\frac{t^{d/2}}{h_d\left(\sqrt{t}\right)},\qquad t>0.
 \eeq
 Next, since $M(G,s)=\G(1+s)$, we obtain from \eqref{E5.3} and relation \eqref{E1.4}
 of Proposition \ref{P1.1a}
 \beq\label{E5.4}
 \int_0^\iy\frac{t^{s/2-1}}{h_d(\sqrt{t})}dt
 =M\left(H_d,\frac{s-d}{2}\right)=M\left(\varTheta_d,\frac{s-d}{2}\right)
 \G\left(1+\frac{s-d}{2}\right).
 \eeq
 Thus, changing the variable in the left-hand side of \eqref{E3.16}
 and replacing the right-hand side of \eqref{E3.16}
 with the right-hand side of \eqref{E5.4},
 we arrive at \eqref{E5.1}.
 The uniform convergence in \eqref{E5.1}
  follows from Corollary \ref{C3.10}.
\end{proof}

The following question  arises naturally: is it possible also
to prove pointwise asymptotic relations for $B_N^*$?
We address this question in the following theorem.

\begin{theorem}\label{T5.2}
Let the sequences
$\left(P_{2N+d}\right)_{N=1}^\iy\in\PPP_d(\be,\g,\de),\,
 \left(u_N\right)_{N=1}^\iy$ and the set $\Om_{N+2,d}$
 be the same as in Theorem \ref{T5.1}.
 In addition, assume that the sequence
 $\left(z^{-d}P_{2N+d}(z)/P_{2N+d}^{(d)}(0)\right)_{N=1}^\iy$
 satisfies the $(r,\be)$-Condition (see Definition \ref{D6.1}).
 Then the following relation holds:
 \beq\label{E5.5}
 \lim_{N\to\iy}\left[\frac{u_N^2 \,\prod_{k=1}^Nx_{k,N,d}^2}{N}\,
 B_N^*\left(\frac{N}{\be^2_N}t,\Om_{N+2,d}\right)\right]
 =\varTheta_d(t),\qquad t\in [0,\iy).
 \eeq
 \end{theorem}
 \begin{proof} The proof consists of three steps.\\
 \textbf{Step 1.}
 We first need certain formulae.
 Assume that $d<\mathrm{Re}\, s<2N+d+1$ and $s-d\ne 2,\,4,\ldots$.
 Then setting $Q_N(u)=G_{2N}(y)=y^{-d}P_{2N+d}(y)/P_{2N+d}^{(d)}(0)$
 and $u_N=\left\vert y_N\right\vert$,
 we obtain from \eqref{E2.10}, \eqref{E2.3}, and \eqref{E2.4}
 \bna\label{E5.7}
 &&u_N^{s-d}-
 L_{N}\left(u^2_N,u^{\frac{s-d}{2}},u Q_N(u)\right)\nonumber\\
 &&=\frac{2i^d\sin((s-d)\pi/2)}{\pi}
 \frac{Q_N\left(u_N^2\right)}{Q_N\left(0\right)}P_{2N+d}^{(d)}(0)
 \bigintss_0^\iy\frac{t^{s-1}}{\left(1+\left(t/u_N\right)^2\right)P_{2N+d}(i t)}dt.
 \ena
 Using now \eqref{E5.7}, \eqref{E3.10}, \eqref{E3.11}, and \eqref{E3.18},
 we arrive at
 \bna\label{E5.8}
 &&g_{N}(s):=
 \frac{u_N^2 \,\prod_{k=1}^Nx_{k,N,d}^2}{N}\,
 \int_0^{\iy}B_N^*\left(\frac{N}{\be^2_N}t,\Om_{N+2,d}\right)t^{(s-d)/2-1}dt \nonumber\\
   &&=u_N^2 \,\be_N^{s-d}\,N^{-1-(s-d)/2}\,\prod_{k=1}^Nx_{k,N,d}^2
 \int_0^{\max \Om_{N+2,d}}B_N^*\left(t,\Om_{N+2,d}\right)t^{(s-d)/2-1}dt\nonumber\\
 &&=\left(\frac{2i^d}{\G\left(1+\frac{s-d}{2}\right)}\right)
 \left(\frac{(N+1)!}{N^{1+(s-d)/2}\G\left(N+1-\frac{s-d}{2}\right)}\right)
 \bigintss_0^\iy\frac{\be_N^{s-d}\,P_{2N+d}^{(d)}(0)\,t^{s-1}}
 {\left(1+\left(t/u_N\right)^2\right)P_{2N+d}(i t)}dt,
 \ena
 where by \eqref{E5.1},
 \beq\label{E5.9}
 \lim_{N\to\iy} g_{N}(s)=M\left(\varTheta_d,\frac{s-d}{2}\right),
 \qquad \mathrm{Re}\,s>d.
 \eeq
 \textbf{Step 2.}
 Our next goal is to find a "good" upper estimate of $\left\vert g_{N}(s)\right\vert$
 in terms of $\vert\mathrm{Im}\,s\vert$. In other words, we need to estimate the absolute values of
 the three factors on the right-hand side of \eqref{E5.8}.

 The first estimate
 \beq\label{E5.10}
 \left\vert\frac{2i^d}{\G\left(1+\frac{s-d}{2}\right)}\right\vert
 < \frac {12e^{(\pi/4)\vert\mathrm{Im}\,s\vert}}
 {\vert\mathrm{Im}\,s\vert^{\frac{\mathrm{Re}\,s-d+1}{2}}},
 \qquad \vert\mathrm{Im}\,s\vert \gtrsim_{\mathrm{Re}\,s,d} 1,
 \eeq
directly follows from the known relation
 \ba
 \lim_{y\to\iy}\G(\sa+iy)e^{(\pi/2)\vert y\vert}\vert y\vert^{1/2-\sa}=\sqrt{2\pi}
 \ea
 (see, e.g., \cite[Eq. 1.18(6)]{ErdI1953}).

 Next, it follows from \eqref{E3.19} that
 \beq\label{E5.11}
 \frac{(N+1)!}{\left \vert N^{1+(s-d)/2}\G\left(N+1-\frac{s-d}{2}\right)\right\vert}
 <\frac{2\G(M)}{\vert\G(M+iy)\vert},\quad M:=N+1-(\mathrm{Re}\,s-d)/2
 \gtrsim_{\mathrm{Re}\,s,d} 1,
 \eeq
 where $y=(-\mathrm{Im}\,s)/2$. Since $\vert\arg(M+iy)\vert\le \pi/2$,
 we can apply Stirling's formula (see, e.g., \cite[Eq. 1.18(2)]{ErdI1953})
 to the quotient of the two gamma functions in the right-hand side of \eqref{E5.11}. Thus,
 \beq\label{E5.12}
 \frac{\G(M)}{\vert\G(M+iy)\vert}
 \le \frac{e^{y\arg(M+iy)}}{\left(\sqrt{1+(y/M)^2}\right)^{M-1/2}}\left(1+C/M\right)
 :=A(M,y)\left(1+C/M\right),
 \eeq
 where $C>0$ is an absolute constant.
 Note that $y\arg(M+iy)\ge 0$,
 so without loss of generality we can assume $y\ge 1$ and $M\ge 1$.

 Furthermore, if $1\le M\le \sqrt{y}$, then
  \beq\label{E5.13}
  A(M,y) \le y^{-M/2+1/4}e^{(\pi/2)y}.
  \eeq
  If $1\le \sqrt{y}<M$, then
  \beq\label{E5.14}
  A(M,y)
  \le e^{y\arg(M+iy)}
  \le e^{(\pi/2)y}e^{-y\,\mathrm{arccot}(\sqrt{y})}
  \le e^{(\pi/2)y-\sqrt{y/2}}.
  \eeq
  It follows from \eqref{E5.13} and \eqref{E5.14} that for any fixed $M_0>0$ and any
  $M\ge 2M_0+1/2$,
  \beq\label{E5.14a}
  A(M,y)
  \lesssim_{M_0} y^{-M_0}e^{(\pi/2)y},
  \qquad  y\ge 1.
  \eeq
  Finally, collecting relations \eqref{E5.11}, \eqref{E5.12}, and \eqref{E5.14a},
   we obtain the second estimate
  \beq\label{E5.15}
  \frac{(N+1)!}{\left \vert N^{1+(s-d)/2}\G\left(N+1+\frac{s-d}{2}\right)\right\vert}
  \lesssim_{M_0} \frac{e^{(\pi/4)\vert\mathrm{Im}\,s\vert}}{\vert\mathrm{Im}\,s\vert^{M_0}},
  \qquad N\gtrsim_{\mathrm{Re}\,s,M_0,d} 1,\quad \vert\mathrm{Im}\,s\vert\ge 1,
  \eeq
  for any fixed $M_0>0$.

  The third estimate for the integral in the right-hand side of \eqref{E5.8}
  follows from the $(r,\be)$-Condition. Indeed, by the assumption of Theorem \ref{T5.2}
  and Definition \ref{D6.1}, there exist constants $r^*=r+d>d$ and
  $\mu(r^*,d)$ such that
  \beq\label{E5.16}
  \be_N^{\mathrm{Re}\,s-d}\left\vert\bigintss_0^\iy\frac{\,P_{2N+d}^{(d)}(0)\,t^{s-1}}
 {\left(1+\left(t/u_N\right)^2\right)P_{2N+d}(i t)}dt\right\vert
 \lesssim_{\mathrm{Re}\,s,d}\left\vert\mathrm{Im}\,s\right\vert^{\mu(\mathrm{Re}\,s,d)}
  e^{-(\pi/2)\left\vert\mathrm{Im}\,s\right\vert},
  \quad \mathrm{Re}\,s=r^*.
  \eeq
  Thus, combining estimates \eqref{E5.10}, \eqref{E5.15}, and \eqref{E5.16} with
  equality \eqref{E5.8}, we obtain
  \beq\label{E5.17}
  \left\vert g_{N}(s)\right\vert
  \lesssim_{\mathrm{Re}\,s,M_0,d}
 \vert\mathrm{Im}\,s\vert^{-M_0-\frac{\mathrm{Re}\,s-d+1}{2}+\mu(\mathrm{Re}\,s,d)},
 \qquad \mathrm{Re}\,s=r^*,
  \eeq
  for any fixed $M_0>0$
  and $N\gtrsim_{r^*,M_0,d} 1, \,\vert\mathrm{Im}\,s\vert\gtrsim_{r^*,d} 1$.
  \vspace{.1in}\\
  \textbf{Step 3.}
  Here, we discuss certain properties of $g_N$ and conclude the proof of the theorem.
  Let us define the function
  \beq\label{E5.18}
  Q_{N,d}(t)
   :=  \frac{u_N^2 \,\prod_{k=1}^Nx_{k,N,d}^2}{N}\,
 B_N^*\left(\frac{N}{\be^2_N}t,\Om_{N+2,d}\right),
   \qquad t\in[0,\iy).
  \eeq
  \begin{proposition}\label{P5.3}
  There exists $r^*>d$ such that the following statements hold:\\
  (a) $Q_{N,d}\in\LL.$\\
  (b) $g_N(s)=M\left(Q_{N,d},\frac{s-d}{2}\right),\,\mathrm{Re}\,s=r^*$.\\
  (c) $\lim_{N\to\iy}M\left(Q_{N,d},\frac{s-d}{2}\right)=M\left(\varTheta_d,\frac{s-d}{2}\right),\,
  \mathrm{Re}\,s=r^*$, and the convergence is uniform on any compact subset
  of the line $\mathrm{Re}\,s=r^*$.\\
  (d) For any $M_1>0$,
  \beq\label{E5.19}
  \sup_{N\in\N} \left\vert g_N(r^*+2i\tau)\right\vert
  \lesssim_{r^*,M_0,d}(1+\vert\tau\vert)^{-M_1},
  \qquad \tau\in \R.
  \eeq
  (e) $M\left(Q_{N,d},(r^*-d)/2+i\cdot\right)\in L_1(\R)$.
  \end{proposition}
  \begin{proof}
  Statement (a) follows from \eqref{E5.18} and Property \ref{P3.8},
  while (b) is a consequence of \eqref{E5.8}.
  In addition, (c) follows from statement (b) and relation \eqref{E5.9}.
  To prove (d), we first note that by statement (c), for a compact subset $K$
  of the line $\mathrm{Re}\,s=r^*$,
  \beq\label{E5.20}
  \sup_{N\in\N} \max_{\tau\in K}  \left\vert g_N(r^*+2i\tau)\right\vert<\iy.
  \eeq
  Then \eqref{E5.19} follows from \eqref{E5.20} and \eqref{E5.17}.
  Finally, (e) follows from (a), (b), and (d).
  \end{proof}
  Next,
  \bna
  &&Q_{N,d}(t)
  =\frac{1}{2\pi}\int_{\R} M(Q_{N,d},(r^*-d)/2+i\tau)\,t^{-(r^*-d)/2-i\tau}d\tau,
  \qquad t\in[0,\iy),\label{E5.21}\\
  &&\varTheta_d(t)
  =\frac{1}{2\pi}\int_{\R} M(\varTheta_{d},(r^*-d)/2+i\tau)\,t^{-(r^*-d)/2-i\tau}d\tau,
  \qquad t\in[0,\iy). \label{E5.21a}
  \ena
  Indeed, inversion formula \eqref{E5.21} follows from Proposition \ref{P1.1b}
  and statements (a) and (e) of Proposition \ref{P5.3}.
  To prove \eqref{E5.21a}, we first note that $\varTheta_d\in \LL$ by Lemma \ref{L4.1} (b).
  Next, $ M(\varTheta_{d},(r^*-d)/2+i\cdot)\in\LL$ as well by relations
  \eqref{E5.9} and \eqref{E5.19}. Hence \eqref{E5.21a} follows from Proposition \ref{P1.1b}.

  Finally, taking account of \eqref{E5.21}, \eqref{E5.21a},
  and statements (b), (c), and (d) of Proposition \ref{P5.3},
  we see that by the Dominated Convergence Theorem,
  \ba
  \lim_{N\to\iy}Q_{N,d}(t)=\varTheta_d(t),\qquad t\in[0,\iy).
  \ea
  Thus, \eqref{E5.5} is established.
 \end{proof}

The following corollary shows that limit relation \eqref{E5.5}
is valid for the Chebyshev polynomials.

\begin{corollary}\label{C5.5}
Let $\be_N=\be_{N,\la}:=2N+d+\la$ and let $P_{2N+d}:=2^{-(2N+d-1+\la)}Q_{2N+d,\la},\,N\in\N,\,\la=0$ or $\la=1$,
 be the normalized Chebyshev polynomial of the first or second kind (see \eqref{E6.9} or \eqref{E6.9a}
 and also Example \ref{Ex1.2}).
 In addition, let $\left(u_{N,\la}\right)_{N=1}^\iy$ be a sequence of positive numbers
 such that $u_{N,\la}$ is not a zero of $Q_{2N+d,\la}$ and
 $\lim_{N\to\iy} \be_{N,\la} u_{N,\la} =\iy$.
Then the following relations hold:
 \bna
 &&\lim_{N\to\iy}
 \left[\frac{u_{N,\la}^2}{N\,2^{2N+d-1+\la}}B_N^*\left(\frac{N}{(2N+d+\la)^2}t,\Om_{N+2,d,\la}\right)\right]
 =\varTheta_d(t),\qquad t\in [0,\iy);\label{E5.22}\\
 &&B_N\left(\frac{N}{(2N+d+\la)^2}t,\Om_{N+2,d,\la}\right)\nonumber\\
 &&=
 \frac{N\,2^{2N+d-1+\la}}{u_{N,\la}^2}
 \varTheta_d(t)\left(\frac{N}{(2N+d+\la)^2}t\right)^N(1+o(1)), \qquad t\in [0,\iy);\label{E5.22a}
 \ena
 as $N\to\iy$, where
 \bna
 &&\Om_{N+2,d,0}:=\{0, u_{N,0}^2\}\bigcup
 \left\{\cos^2\left(\frac{(2k-1)\pi}{4N+2d}\right):1\le k\le N\right\},\label{E5.24}\\
 &&\Om_{N+2,d,1}:=\{0, u_{N,1}^2\}\bigcup
 \left\{\cos^2\left(\frac{k\pi}{2N+d+1}\right):1\le k\le N\right\}.\label{E5.24a}
 \ena
\end{corollary}
\begin{proof}
The sequence $\left(2^{-(2N+d-1+\la)}Q_{2N+d,\la}\right)_{N=1}^\iy$ belongs to $\PPP_d(\be,\g,\de)$
by Example \ref{Ex1.2} and, in addition,
the sequence
 $\left(t^{-d}Q_{2N+d,\la}(t)/Q_{2N+d,\la}^{(d)}(0)\right)_{N=1}^\iy$
 satisfies the $(r,\be)$-Condition by Lemma \ref{L6.6}.
 Note also that $\prod_{k=1}^Nx_k^2=2^{-(2N+d-1+\la)}$.
 Therefore, \eqref{E5.22} follows from Theorem \ref{T5.2} while \eqref{E5.22a} is an immediate consequence
 of \eqref{E5.22} and \eqref{E3.2}.
\end{proof}

\begin{remark}\label{R5.7a}
It follows from \eqref{E5.22a} that
\ba
\lim_{N\to\iy}B_N\left(\frac{N}{(2N+d+\la)^2}t,\Om_{N+2,d,\la}\right)
=4\lim_{N\to\iy}N B_N\left(t,\frac{(2N+d+\la)^2}{N}\Om_{N+2,d,\la}\right)=0
\ea
for any $t\ne 0$, i.e., the $B$-spline converges to the delta function as $N\to \iy$ (cf. Theorem \ref{T1.2} (iii)
and Example \ref{Ex1.3a}).
In addition, note that all knots in Theorem \ref{T5.2} and Corollary \ref{C5.5} are nonnegative,
so condition 1. of Example \ref{Ex1.1aa} is not satisfied. That is why any comparisons between
 \eqref{E5.5} and \eqref{E5.22} with \eqref{E1.1c0} are not possible.
\end{remark}

\begin{remark}\label{R5.8}
We conjecture that relations like \eqref{E5.22} and \eqref{E5.22a} are also valid for the normalized
Gegenbauer and Hermite polynomials and for the polynomials with equidistant zeros (see Examples \ref{Ex1.2}--\ref{Ex1.4}).
\end{remark}

It turns out that relation \eqref{E5.22a} can be reformulated
for a $B$-spline with knots in the interval $[-1,\iy)$.

\begin{corollary}\label{C5.6}
Let
\bna
&&\Om_{N+2,d,0}^{*}:=\{-1, 2u_{N,0}^2-1\}
 \bigcup \left\{\cos\left(\frac{(2k-1)\pi}{2N+d}\right):1\le k\le N\right\},\label{E5.25}\\
 &&\Om_{N+2,d,1}^{*}:=\{-1, 2u_{N,1}^2-1\}
 \bigcup \left\{\cos\left(\frac{2k\pi}{2N+d+1}\right):1\le k\le N\right\},\label{E5.25a}
\ena
 where $\left(u_{N,\la}\right)_{N=1}^\iy$ is a sequence of positive numbers
 such that $\lim_{N\to\iy} N \,u_{N,\la} =\iy$ and the
 two sets in the right-hand sides of
 \eqref{E5.25} and \eqref{E5.25a} are disjoint.
 Then
 \bna\label{E5.26}
 &&
 B_{N}\left(\frac{2N}
 {(2N+d+\la)^2}t-1,\Om_{N+2,d,\la}^{*}\right)\nonumber\\
 &&=\frac{N\,2^{2N+d-1+\la}}{2u_{N,\la}^2} \varTheta_d(t)\left(\frac{2N}
 {(2N+d+\la)^2}t\right)^{N}\left(1+o(1)\right),
 \qquad  t\in [0,\iy),\quad \la=0,\,1,
 \ena
 as $N\to\iy$.
 \begin{proof}
 By elementary calculations,
 \beq\label{E5.27}
 B_{N}\left(y,\Om_{N+2,d,\la}\right)
 =2B_{N}\left(2y-1,\Om_{N+2,d,\la}^{*}\right),
 \eeq
 where sets $\Om_{N+2,d,\la}$ and $\Om_{N+2,d,\la}^{*}$ are defined by
 \eqref{E5.24}--\eqref{E5.25a}.
 Thus, \eqref{E5.26} follows from \eqref{E5.22a} and \eqref{E5.27}.
\end{proof}
\end{corollary}

\begin{remark}\label{R5.7}
If we choose $u_{N,1}=1$, then the set $\Om_{N+2,1,1}^{*}\subseteq [-1,1]$ defined in \eqref{E5.25a} consists
of all zeros of the polynomial $(x^2-1)U_N(x)\in\PP_{N+2}$.
Hence relation  \eqref{E5.26} for $d=1$ and $\la=1$ describes the asymptotic behavior of the perfect $B$-spline
from Example \ref{Ex1.3a} in a right neighborhood of the point $-1$ (cf. \eqref{E1.1c2} and \eqref{E1.1c3}).
In addition, note that if we choose $u_{N,0}=1$, then the set $\Om_{N+2,0,0}^{*}\subseteq [-1,1]$ defined in \eqref{E5.25} consists
of all zeros of the polynomial $(x^2-1)T_N(x)\in\PP_{N+2}$.
\end{remark}
\noindent
\textbf{Acknowledgements.} We are grateful to the anonymous referees
 for valuable suggestions. In addition, we thank Maciej Rzeszut and Micha\l\, Wojciechowski
 for sharing reference \cite{RW2024}.

\end{document}